\documentclass{amsproc}

 \newtheorem{teor}{Theorem}[section]

 \newtheorem{prop}[teor]{Proposition}
 \theoremstyle{definition}
 
 \theoremstyle{remark}

  \newtheorem{obs}{Remark}[section]
\numberwithin{equation}{section}

\usepackage{color}
\usepackage{graphics}
\usepackage{graphicx}
\usepackage{epsfig}
\usepackage{amssymb}
\usepackage{verbatim}
\usepackage{graphics,amsmath,amsfonts,amssymb} 
\usepackage{epsfig,verbatim,multicol,hyperref}

\begin{document}

\title[Algebraic and qualitative aspects of quadratic V. fields and OP]{Algebraic and qualitative aspects of quadratic vector fields related with classical orthogonal polynomials}

\author[P. Acosta-Hum\'anez]{Primitivo B, Acosta-Hum\'anez}
\address[P. Acosta-Hum\'anez]{Instituto Superior de Formaci\'on Docente Salom\'e Ureña - ISFODOSU, Recinto Emilio Prud-Homme, Santiago de los Caballeros, Dominican Republic\\ Facultad de Ciencias B\'asicas y Biomedicas, Universidad Sim\'on Bol\'ivar, Barranquilla - Colombia}
\email{primitivo.acosta-humanez@isfodosu.edu.do}

\author[M. Campo D.]{Maria Campo Donado}
\address[M. Campo D.]{Facultad de Ciencias B\'asicas, programa de matem\'atica universidad del Atl\'antico, Barranquilla - Colombia}
\email{mcampodonado@mail.uniatlantico.edu.co}

\author[A. Reyes L.]{Alberto Reyes Linero}
\address[A. Reyes L.]{Programa de matem\'aticas- Universidad del Atl\'antico, Barranquilla - Colombia}
\email{areyeslinero@mail.uniatlantico.edu.co}

\author[J. Rodr\'iguez C.]{Jorge Rodr\'{i}guez Contreras}
\address[J. Rodr\'iguez C.]{Departamento de de matem\'atica y estad\'istica Universidad del Norte \& Programa de matem\'aticas- Universidad del Atl\'antico, Barranquilla - Colombia}
\email{jrodri@uninorte.edu.co - jorgelrodriguezc@mail.uniatlantico.edu.co}

\dedicatory{To Isaias Acosta-Hum\'anez Morillo with occasion of his first aniversary}
\maketitle
      
\begin{abstract}
This paper is a sequel of the reference \cite[\S 4.2, p.p. 1782--1783]{almp}, in where some families of quadratic polynomial vector fields related with orthogonal polynomials were studied. We extend such results that contain some details related with differential Galois Theory as well the inclusion of Darboux theory of integrability and qualitative theory of dynamical systems.\\


\end{abstract}

\section*{Introduction}

To study any process of variation with respect the time, the theory of dynamical systems has been developed, which is also endowed with algebraic and qualitative techniques, among others.  Although, in general case, it is not possible to find the solution of a differential equation that models an specific process, we can identify geometric structures influencing over qualitative properties such as; stability, invariant sets attractors, among others, see \cite{DM,120,121,119,PK,VDP} for further details. In the algebraic sense, E. Picard and E. Vessiot introduced an approach to study linear differential equations based on the Galois theory for polynomials (see \cite{HE}), which is known as differential Galois theory or also known Picard-Vessiot theory, see \cite{aclm,acpe2,crha,vasi} for further details. Also G. Darboux introduced an algebraic theory to analyse the integrability of polynomial vector fields, which is known as Darboux theory of integrability, see  \cite{LZ} and references therein. The final ingredient of this paper corresponds to orthogonal polynomials, see \cite{chihara,ismail}, which are very important in both theoretical and applied mathematics giving contributions to random matrices, approximation theory, trigonometric series and especially differential equations, among others. \\

Concerning to applications of differential Galois theory to dynamical systems, in \cite{mo,mora} were presented techniques to determine the non-integrability of Hamiltonian systems can be found in references \cite{almp2,almp,acmowe,acpan,mo,mora}, while in \cite{almp2,almp} were presented techniques to study planar polynomial vector fields. In the same way, applications to Quantum Mechanics can be found in \cite{acbook,acmowe}. Combinations of algebraic and qualitative techiques to study planar vector fields were presented in \cite{ARR,ARRg}. This paper is a sequel of \cite{almp}, in particular is an extension  of the section \S 4.2. We follow the same structure of the papers \cite{ARR,ARRg} concerning the algebraic and qualitative techniques to study the polynomial vector fields. We recall that for algebraic analysis, differential Galois theory and Darboux integrability, we consider vector fields over the complex numbers, while for qualitative analysis we consider the vector fields over the real numbers.

\section{Preliminaries}
In this section we present the basic theoretical background needed to understand the rest of the paper.
\subsection{Classical Orthogonal Polynomials}
The main object of study in this work are quadratic polynomial differential systems associated to classical orthogonal polynomials.
In particular we focus in the sequences of classical orthogonal polynomials of hypergeometric type, that is, orthogonal polynomials satisfying the differential equation
\begin{equation}
\rho(x)y''+\tau(x)y'+\lambda y=0 \label{HE},\end{equation}
where $\rho(x)$, $\tau(x)$ are polynomials and $\lambda$ depending on $n$ are given in the next table:

\begin{center}
	\begin{tabular}{| c | c | c | }
		\hline
		$\rho(x)$           & $\tau(x)$                          & $\lambda_{n}$            \\ \hline
		$1-x^{2}$           & $\beta-\alpha-(\alpha+\beta+2)x$      & $n(n+1+\alpha+\beta)$    \\ \hline
		$1-x^{2}$           &    $-2x $                             & $n(n+1)$                 \\ \hline
		$1-x^{2}$           & $-x $                                 & $n^{2}$                  \\ \hline
		$1-x^{2}$           & $ -3x $                               &$ n(n+2)$                 \\ \hline
		$1-x^{2}$           &  $-(2\alpha+1)x$                      & $n(n+1+2\alpha)$           \\ \hline
		$x $                &   $\alpha+1-x $                       &$ n $                     \\ \hline
		$x $                &  $ 1-x$                               &$ n $                     \\ \hline
		$ 1$                &  $ -2x$                               & $2n $                    \\
		\hline 
	\end{tabular}
\end{center}

Moreover, it is well known that classical orthogonal polynomials can be obtained by Rodrigues formula, see \cite{chihara,ismail}. In a general form, the constant $\lambda_n$ can be obtained as follows. $$ \lambda_{n}=-n\left( \tau'+\frac{n-1}{2} \rho ''(x) \right).$$
Thus, the object of study becomes the differential system

$$\begin{array}{rl}
\frac{d\nu}{dt}&=\frac{\lambda_{n}}{\mu}\rho+(\rho'-\tau)\nu+\mu \nu^2\\
\frac{dx}{dt}&=\rho
\end{array}$$
and its associated foliation becomes to
$$\frac{d\nu}{dx}=\frac{\lambda_{n}}{\mu}+\frac{\rho'-\tau}{\rho}\nu+\frac{\mu}{\rho} \nu^2.$$
 We claim that $\mu\neq 0$ because we are studying quadratic polynomial vector fields. 
  
%
%
%
%
%
%
%
%
%
%

\subsection{Critical Points}
We recall that a real vector field $\chi$ is a function of $C^r$  class where $r\in \mathbb{N} \cup {\infty,\omega}$  (if $r=\omega$ we say that the function is analytic). Moreover,  $\chi:\Delta \longrightarrow \mathbb{R}$ and $\Delta$ is an open subset of $\mathbb{R}$. For instance the differential system associated to the vector field $\chi$ is given by $\dot{x}=\chi(x)$. Now, based on the references \cite{DM,PK}, we present the classification of some critical points used in the main results of this paper. The following theorem is concerning to hyperbolic critical points.

\begin{teor}\label{teo hiper}
	Let $(0,0)$ be an isolated singular point of the vector field $X$ associated to
	\begin{equation}
	\label{sis hiper}
	\begin{array}{rl}
	\dot{v}&=av+bx+ A(v,x),\\
	\dot{x}&=cv+dx+ B(v,x),
	\end{array}
	\end{equation}
	where $A$ and $B$ are analytic in a neighborhood of the origin with $A(0,0)=B(0,0)=DA(0,0)=DB(0,0)=0$. Let $\lambda_{1}$ and $\lambda_{2}$ be an eigenvalue of the linear part $DX(0,0)$ of the system at the origin. Then the following statements hold.
	
	\begin{itemize}
		\item If $\lambda_{1}$ and $\lambda_{2}$ are real and $\lambda_{1}\lambda_{2}<0$, then $(0,0)$ is a saddle. If we denote by  $E_{1}$ and $E_{2}$ the eigenspaces of respectively $\lambda_{1}$ and $\lambda_{2}$ then one can find two invariant analytic curves, tangent respectively to $E_{1}$ and $E_{2}$ at $0$, on one of which points are attracted towards the origin, and on one of which points are repelled away from the origin .On these invariant curves $X$ is $C^{\omega}-$linearizable. There exist a $C^{\infty}$ coordinate change transforming 
		\eqref{sis hiper} into one of the following normal forms:
				$$\begin{array}{rl}
		\dot{v}&=\lambda_{1}v,\\
		\dot{x}&=\lambda_{2}x,
		\end{array}$$
				in the case $\lambda_{1}/\lambda_{2}\in \mathbb{R}\setminus\mathbb{Q}$, and
		$$\begin{array}{rl}
		\dot{v}&=v(\lambda_{1}+f(v^{k}x^{l})),\\
		\dot{x}&=x(\lambda_{2}+g(v^{k}x^{l})),
		\end{array}$$
				in the case $\lambda_{1}/\lambda_{2}=-k/l\in \mathbb{Q}$with $k,l\in \mathbb{N}$ and where $f,g$ are functions $C^{\infty}$. All systems \ref{sis hiper} are $C^{0}$-conjugate to
		$$\begin{array}{rl}
		\dot{v}&=v,\\
		\dot{x}&=-x.
		\end{array}$$
		
		\item If $\lambda_{1}$ and $\lambda_{2}$ are real with $|\lambda_{2}|\geq|\lambda_{1}|$ and $\lambda_{1}\lambda_{2}>0$, then $(0,0)$ is a node. If $\lambda_{1}>0$ (Respectively $<0$) then it is repelling or unestable (respectively attracting or stable). There exist a $C^{\infty}$ coordinate change transforming \ref{sis hiper} into
		$$\begin{array}{rl}
		\dot{x}&=\lambda_{1}x,\\
		\dot{y}&=\lambda_{2}y,
		\end{array}$$
		in case $\lambda_{1}/\lambda_{2}\not\in \mathbb{N}$, and into
		$$\begin{array}{rl}
		\dot{x}&=\lambda_{1}x,\\
		\dot{y}&=\lambda_{2}y+\eta x^{m},
		\end{array}$$
				for some $\eta=0$ or $1$, in case  $\lambda_{2}=m\lambda_{1}$ with $m\in \mathbb{N}$ and $m>1$. All systems are $C^{0}-$conjugate to
		$$\begin{array}{rl}
		\dot{x}&=\eta x,\\
		\dot{y}&=\eta y,
		\end{array}$$
		with $\eta= \pm 1$ and $\lambda_{1}\eta>0$.
		
		\item If $\lambda_{1}=\alpha+\beta i$ and $\lambda_{2}=\alpha-\beta i$ with $\alpha,\beta\neq0$ then $(0,0)$ is a "strong" focus. If $\alpha>0$ (respectively $\alpha<0$), it is repelling or unstable (respectively attracting or stable). There exists a $C^{\infty}$ coordinate change transforming \ref{sis hiper} into
		$$
		\begin{array}{rl}
		\dot{x}&=\alpha x+\beta y,\\
		\dot{y}&=-\beta x+\alpha y.
		\end{array}
		$$
		All systems \ref{semi hiper} are  $C^{0}-$conjugado to
		$$\begin{array}{rl}
		\dot{x}&=\eta x,\\
		\dot{y}&=\eta y,
		\end{array}$$
		with $\eta= \pm 1$ and $\alpha\eta>0$.
		
		\item If $\lambda_{1}=\beta i$ and $\lambda_{2}=-\beta i$ with $\beta\neq0$, then $(0,0)$ is a linear center topologically, a \emph{weak} focus or a center. 
	\end{itemize}
\end{teor}

The following theorem corresponds to Semi-hyperbolic critical points.

\begin{teor}

	Let $(0,0)$ be an isolated singular point of the vector field $X$ given by
	\begin{equation}
	\label{semi hiper}
	\begin{array}{rl}
	\dot{x} & =A(x,y)\\
	&\\
	\dot{y}& =\lambda y+B(x,y)
	\end{array}
	\end{equation}
	
	where $A$ and $B$ are analytic in a neighborhood of a origin with  $A(0,0)=B(0,0)=DA(0,0)=DB(0,0)=0$ and $\lambda>0$. Let $y=f(x)$ be the solution of equation $\lambda y+B(x,y)=0$ in a neighborhood of the point $(0,0)$, and supose that the function $g(x)=A(x,f(x))$ has the expression $g(x)=a_{m}x^{m}+o(x^{m})$ where $m\geq2$ and $a_{m}\neq0$. Then there always exists an invariant analytic curve, called the strong unstable manifold, tangent at $0$ to the $\mathbf{0}$ to the $y-$axis, on which $X$ is analytically conjugate to
	$$\frac{dx}{dt}=\lambda x;$$
	it represents repelling behavior since $\lambda>0$. Moreover the following statements hold.
	\begin{enumerate}
		\item [(i)] If $m$ id odd and $a_{m}<0$ then $(0,0)$ is a topologycal saddle. tangent to the $x-$axis there is a unique invariant $C^{\infty}$ curve, called the center manifold, on which $X$ is $C^{\infty}$-conjugate to
		$$\dot{x}=-x^{m}(1+ax^{m-1}),$$
		for some $a\in \mathbb{R}$. 
		
		If this invariant curve  is analytic, then on it $X$ is $C^{\infty}$-conjugate to
		$$\begin{array}{rl}
		\dot{x}&=-x^{m}(1+ax^{m-1}),\\
		\dot{y}&=\lambda y,
		\end{array}$$
		and is $C^{0}$-conjugate to
		$$\begin{array}{rl}
		\dot{x}&=-x,\\
		\dot{y}&= y.
		\end{array}$$

		\item [(ii)] if $m$ is odd and $a_{m}>0$, the origin is a unstable topological node. Every point not belonging to the strong unstable manyfold lies  on an invariant $C^{\infty}$ curve called a center manifold , tangent to the x-axis at the origin, and on which $X$ is a $C^{\infty}$-conjugate to
		$$\dot{x}=x^{m}(1+ax^{m-1}),$$
		for some $a\in \mathbb{R}$. All these center manifold are mutually infinitely tangent to each othe, and hence at most one of them ca be analytic, in which case $X$ is $C^{\infty}$-conjugate to
		$$\begin{array}{rl}
		\dot{x}&=x^{m}(1+ax^{m-1}),\\
		\dot{y}&=\lambda y,
		\end{array}$$
		and $C^{0}$-conjugate to
		$$\begin{array}{rl}
		\dot{x}&=x,\\
		\dot{y}&= y.
		\end{array}$$
		
		\item [(iii)] If $m$ is even, then $(0,0)$ is a saddle node, that is a singular point whose neigborhood is the union of one parabolic and two hiperbolic sectors. Modulo changing $x$ into $-x$, we suppose that $a_{m}> 0$. Every point to the right of the strong unstable manifold (side $x> 0$) lies on a invariant $C^{\infty}$ curve, called a center manifold, tangent to the x-axis at the origin, and on which case $X$ is a $C^{\infty}$-conjugate to
		\begin{equation}\nonumber
		\dot{x}=x^{m}(1+ax^{m-1}),
		\end{equation}
		for some $a \in \mathbb{R}$. All these center manifold coincide on the side $x\leq0$ and are hence infinitely tangent at the origin. At most one of these center manifolds can be analytic, in which case
		$X$ is $C^{\infty}$-conjugate to
		$$\begin{array}{rl}
		\dot{x}&=x^{m}(1+ax^{m-1}),\\
		\dot{y}&=\lambda y,
		\end{array}$$
		and is $C^{0}$-conjugate to
		$$\begin{array}{rl}
		\dot{x}&=x^{2},\\
		\dot{y}&=\lambda y.
		\end{array}$$
	\end{enumerate}
\end{teor}
The following theorem is concerning to Nilpotent singular points.
\begin{teor}
	let $(0,0)$ be an isolated singular point of the vector field $X$ given by  
	$$\begin{array}{rl}
	\dot{x}&=x+A(x,y),\\
	\dot{y}&=B(x,y),
	\end{array}$$
	where $A$ and $B$ are analytic in a neighborhood of the point $(0,0)$ and also $j_1A(0,0)=j_1B(0,0)=0$. Let $y=f(x)$ be the solution of the equtions $y+A(x,y)=0$ in a neighborhood of the point $(0,0)$, and consider $F(x)=B(x,f(x))$ and $G(x)=(\partial A/\partial v+\partial B/\partial x)(x,f(x))$. then the foollowing holds:
	\begin{itemize}
		\item [(i)] If $F(x)\equiv G(x)\equiv0$, then the phase portrait of $X$ is given by \ref{nilpotentes}$a$.
		\item [(ii)] Si $F(x)\equiv 0$ and $G(v)=bx^{n}+ o(x^{n})$ with $n \in \mathbb{N}$, $n\geq1$ and $b\neq0$, then the phase portrait of $X$ is given by \ref{nilpotentes}$b$ o $c$.
		\item [(iii)] If $G(v)\equiv0$ and $F(x)=ax^{m}+ o(x^{m})$ with $m \in \mathbb{N}$, $m\geq1$ and $a\neq0$, then
		\begin{itemize}
			\item If $m$ is odd and $a>0$, then the origin is a saddle (\ref{nilpotentes}$d$) and if $a<0$, then it is a center or focus ( \ref{nilpotentes}$e-f$).
			\item If $m$ is even the origin of $X$ is a cusp (\ref{nilpotentes}$h$).
		\end{itemize}
		\item [(iv)] If $F(x)=ax^{m}+ o(x^{m})$ and $G(x)=bx^{n}+ o(x^{n})$ with $m,n \in \mathbb{N}$, $m\geq1$, $n\geq1$ and $a\neq0$, $b\neq0$, then we have
		\begin{itemize}
			\item If $m$ is even,  and
			\begin{itemize}
				\item $m<2n+1$, then the origin of is a cusp \ref{nilpotentes}$h$
				\item $m>2n+1$, then the origin is a saddle-node \ref{nilpotentes}$i$ or $j$
			\end{itemize}
			
			\item If $m$ is odd and $a>0$ then the origin is a saddle \ref{nilpotentes}$d$.
			\item If $m$ is odd, $a<0$  and
			\begin{itemize}
				\item  Either $m<2n+1$, or $m=2n+1$ and $b^{2}+4a(n+1)<0$, then the origin is a center or focus (figure \ref{nilpotentes}$e$, $g$).
				\item If $n$ is odd and either $m>2n+1$, or $m=2n+1$ and $b^{2}+4a(n+1)\geq0$ then the phase portrait of the origin of $X$ consist of one hyperbolcic and one ellyptic as in figure (\ref{nilpotentes}$k$).
				\item $n$ is even and either  $m>2n+1$, or $m>2n+1$ and  $b^{2}+4a(n+1)\geq0$ then the origin of $X$ is a node as in figure \ref{nilpotentes}$l$, $m$. The node is attracting if $b<0$ and repelling if $b>0$.
			\end{itemize}
		\end{itemize}
	\end{itemize}
\end{teor}

\begin{figure}[h]
	\begin{center}
		\includegraphics[width=10cm]{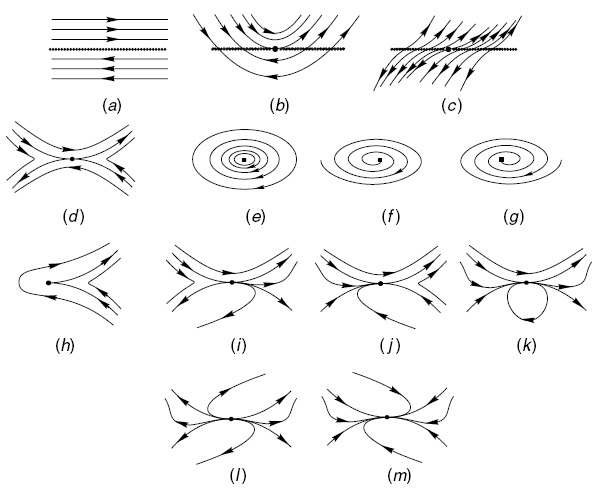}\\
		\caption{Portraits of phase for \ref{sis retratos2}, \cite{DM}}\label{nilpotentes}
	\end{center}
\end{figure}

For complete study of these theorems see \cite{DM}.

\subsection{Invariants Curves}
Let be the differential polynomial complex system
\begin{equation}
\begin{array}{rl}\label{S}
\dot{x}&=P(x,y),\\
\dot{y}&=Q(x,y),
\end{array}
\end{equation}
and $m=max\{deg P, deg Q\}$.

\begin{teor}\label{prop de darboux}
	
	Suppose that a $\mathbb{C}-$polynomial system (\ref{S}) of degree $m$ admits $p$ irreducible invariant algebraic curves $f_{i}=0$ with cofactors $K_{i}=1,2,...,p$;  $q$ exponential factors $exp(g_{i}/h_{i})$ with cofactors $L_{j}$, $j=1,2,...,q$, and $r$ independent singular points $(x_k,y_k)\in\mathbb{C}^2 $ such that $f_i(x_k,y_k)\neq 0)$ then if there exits $\lambda_{i}, \mu_{j}\in\mathbb{C}$ no not all zero such that
	$$\sum_{i=0} ^{p}\lambda_{i}K_{i}+\sum_{j=0}^{q}\mu_{j}L_{j}=-s$$
	
	for some $s\in\mathbb{C}\backslash \{0\}$, then the (multivalued) function
	$$f_{1}^{\lambda_{1}}...f_{p}^{\lambda_{p}}F_{1}^{\mu_{1}}...F_{q}^{\mu_{q}}e^{st}$$
	is an invariant of system (\ref{S})
\end{teor}

For a complete version if this theorem see \cite{DM}[\S 8, p.p. 219].

The following theorems concern to singular points at infinity, where $x=\frac{X}{Z}$ and $y=\frac{Y}{Z}$.

\begin{teor}\label{icp1}
	The critical points at infinity for the mth degre polynomial system (\ref{S}) occur at the points $(X,Y,0)$ over the equator of the Poincar\'e sphere, being  $X^2+Y^2=1$ and $$ XQ_m(X,Y)-YP_m(X,Y)=0.$$
\end{teor}

\begin{teor}
The flow defined in a neighborhood of any critical point of (\ref{S}) (with mentioned change of variable) over the equator of the Poincar\'e sphere $S^2$, except the points $(0,\pm 1,0)$, is topologically equivalent to the flow defined by the system:

\begin{equation}
\begin{array}{rl}\label{S1}
\pm \dot{y}&=yz^mP(\frac{1}{z},\frac{y}{z})-z^mQ(\frac{1}{z},\frac{y}{z}),\\
\pm \dot{z}&=z^{m+1}P(\frac{1}{z},\frac{y}{z}),
\end{array}
\end{equation}
being the signs determined by the flow on the equator of $S^2$ such as was determined in Theorem \ref{icp1}. Similarly, the flow defined by  (\ref{S}) (with the mentioned change of variable) in a neighborhood of any critical point of  (\ref{S}) on the equator of $S^2$ except the points $(\pm 1,0,0)$ is topologically equivalent to the flow defined by the system:

\begin{equation}
\begin{array}{rl}\label{S2}
\pm \dot{x}&=xz^mQ(\frac{x}{z},\frac{1}{z})-z^mP(\frac{x}{z},\frac{1}{z}),\\
\pm \dot{z}&=z^{m+1}P(\frac{1}{z},\frac{y}{z}),
\end{array}
\end{equation}
\end{teor}

the signs being determined by the flow on the equator of $S^2$ as determined in the theorem (\ref{icp1}).

This theory can be study in detail on \cite{DM,PK}.

\section{Main Results}
In this section we set the main results of the paper.
We start presenting some results of orthogonal polynomials theory from a Galoisian point of view. The following proposition relates the classical Galois theory with orthogonal polynomials.

\begin{prop}
	If $P_{n}$ an orthogonal polynomial, then for the splitting field of the polynomial $P_{n}(x)$ over $\mathbb{R}$, ($\mathbb{R}\{P_{n(x)}\}$); we have that
	$\mathbb{R}\{P_{n(x)}\}=\mathbb{R}$.
\end{prop}
\begin{proof}
	Due to the roots $\alpha_{1}, \ldots,\alpha_{n}$ of any orthogonal polynomial $P_{n}$ of degree $n$ are real and distinct, then 
	$$\mathbb{R}\{P_{n}\}=\mathbb{R}[\alpha_{1},...,\alpha_{n}].$$
	Taking the integral domain
	$\mathbb{R}[\alpha_{1}]$. By definition we have that $$\mathbb{R}[\alpha_{1}]=\{f(\alpha_{1})/ f(x)\in \mathbb{R}[x]\}.$$
	Thus $f(\alpha_{1})\in \mathbb{R}$. In this way $\mathbb{R}[\alpha_{1},...,\alpha_{n}]=\mathbb{R}$.
\end{proof}

\begin{obs}
	From the previous proposition we can notice that if we take as base field the real members, then the splitting field of any orthogonal polynomial is again the real numbers. That is, the extension $L=\mathbb{R}\{P_{n}\}=\mathbb{R}$ and therefore the Galois group of the polynomial is $G(L\setminus \mathbb{R})=\{f: f(x)=x, \forall x\in \mathbb{R}\}=Id.$
\end{obs}
The following proposition appears in \cite[\S 4.2]{almp} and it is included, jointly with the proof, for completeness.
\begin{prop}\label{propporic}
	If we consider two polynomials $\rho(x)$, $\tau(x)$ and the parameter $\lambda_{n}$ from the previous table, then for any $\mu$ the Riccati type differential equation 
	\begin{equation}
	\label{ricati}
	\frac{dv}{dx}=\frac{\lambda_{n}}{\mu}+\frac{\rho'-\tau}{\rho}v+\frac{\mu}{\rho}v^{2},
	\end{equation}
	can be transform into the hypergeometric type equation (\ref{HE}).
	\begin{equation}
	\rho(x) y''+K_{1}P_{1}y'+\lambda_{n}y=0\nonumber
	\end{equation}
\end{prop}
\begin{proof}
	Making the change of variable $w=\mu v$ we obtain
	$$\begin{array}{rl}
	\frac{dw}{dx}&=\mu\frac{dv}{dx}\\
	&\\
	&=\lambda_{n}+\frac{\rho'-\tau}{\rho}\mu v+\frac{1}{\rho}\mu^{2}v^{2}\\
	&\\
	&=\lambda_{n}+\frac{\rho'-\tau}{\rho}w+\frac{1}{\rho}w^{2},
	\end{array}$$
	obtaining the differential equation
	$$\frac{dw}{dx}=\lambda_{n}+\frac{\rho'-\tau}{\rho}w+\frac{1}{\rho}w^{2}$$
	Now if we take $w=-\rho\frac{y'}{y}$, then
	\begin{align}\label{2trans}
	y\frac{dw}{dx}+w\frac{dy}{dx}=-\rho' y'-\rho y''.
	\end{align}
	On the other hand,
	$$\begin{array}{rl}
	y\frac{dw}{dx}+w\frac{dy}{dx}&=y\left[\lambda_{n}-\frac{\rho'-\tau}{\rho}\left(\rho\frac{y'}{y}\right)+\frac{1}{\rho}\left(\rho\frac{y'}{y}\right)^{2}+\right]+\left(\rho\frac{y'}{y}\right)y'\\
	&\\
	&=y\lambda_{n}-\rho'y'+\tau y'.
	\end{array}$$
	This is,
	\begin{align}\label{3trans}
	y\frac{dw}{dx}+w\frac{dy}{dx}=y\lambda_{n}-\rho'y'+\tau y'.
	\end{align}
	
	Now by (\ref{2trans}) and (\ref{3trans}), we have
	$$\begin{array}{rl}y\lambda_{n}-\rho'y'+\tau y'&=-\rho' y'-\rho y'',\\
	&\\
	\rho y''+\tau y'+\lambda_{n}y &=0.
	\end{array}$$
\end{proof}

In this way we can associate a polynomial system in the plane to each family of classical orthogonal polynomials as follows:

\begin{center}
	\begin{tabular}{| c | c | c |  }
		\hline
		Family                       & $\dot{v}$           & $\dot{x}$                                   \\ \hline
		$P_{n}^{(\alpha,\beta)}$      & $\frac{\lambda_{n}}{\mu}(1-x^{2})+(\alpha-\beta+(\alpha+\beta)x)v+\mu v^{2}$   & $1-x^{2}$         \\ \hline
		$P_{n}$                       & $\frac{\lambda_{n}}{\mu}(1-x^{2})+\mu v^{2}$ &  $1-x^{2}$                                           \\ \hline
		$T_{n}$                       & $\frac{\lambda_{n}}{\mu}(1-x^{2})-xv+\mu v^{2}$  & $1-x^{2}$                                               \\ \hline
		$U_{n}$                       & $\frac{\lambda_{n}}{\mu}(1-x^{2})+xv+\mu v^{2}$ &$1-x^{2}$                                             \\ \hline
		$C_{n}^{(\alpha)}$            & $\frac{\lambda_{n}}{\mu}(1-x^{2})+(2\alpha-1)xv+\mu v^{2}$  &  $1-x^{2}$                            \\ \hline
		$L_{n}^{(\alpha)}$            &   $\frac{\lambda_{n}}{\mu}x+(-\alpha+x)v+\mu v^{2}$  &  $x $       \\ \hline
		$L_{n}$                       &   $\frac{\lambda_{n}}{\mu}x+xv+\mu v^{2}$   &  $x $                                                  \\ \hline
		$H_{n}$                       &  $\frac{\lambda_{n}}{\mu}+2xv+\mu v^{2}$   &  1                                                \\
		\hline
	\end{tabular}
\end{center}

The following theorem appears in \cite[\S 4.2]{almp} and it is included, jointly with the proof, for completeness.
\begin{teor}
	Let $\rho(x)$, $\tau(x)$ and $\lambda_{n}$ as in the previous proposition. For any $\mu\neq 0$, The quadratic polynomial vector field corresponding to the system
	\begin{equation}
	\label{sispolino}
	\begin{array}{rl}
	\frac{dv}{dt}&=\frac{\lambda_{n}}{\mu}\rho+(\rho'-K_{1}P_{1})v+\mu v^{2},\\
	&\\
	\frac{dx}{dt}&=\rho
	\end{array}
	\end{equation}
	has an invariant algebraic curve of the form $\mu vP_{n}(x)+\rho(x)P'_{n}(x)=0$, where $P_{n}$ is any classical orthogonal polynomial associated to $\rho(x)$, $\tau(x)$ and $\lambda_{n}$
\end{teor}

\begin{proof}
	The differential equation associated with the polynomial system (\ref{sispolino}) is:
	$$\frac{dv}{dx}=\frac{\lambda_{n}}{\mu}+\frac{\rho'-\tau}{\rho}v+\frac{\mu}{\rho}v^{2}$$
	which, by Proposition \ref{propporic} can be transformed in the hypergeometric equation (\ref{HE}) and for each $n\in\mathbb{Z}_+$, we have the solution $y_{n}=P_{n}$ , which is some classical orthogonal polynomial associated with functions $\rho(x)$, $\tau(x)$ and the parameter $\lambda_{n}$.
	\begin{align}\label{sol}
	\rho(x) P''_{n}+\tau P'_{n}+\lambda_{n}P_{n}=0
	\end{align}
	
	Let $X$ be the vector field associated with the differential system (\ref{sispolino})
	Now, for $n$ fixed, we consider the polynomial  $f(v,x)=\mu vP_{n}(x)+\rho P'_{n}(x)$ and we show that it is irreducible and satisfies $Xf=Kf$, where $K$ is the cofactor of the invariant curve $f=0$.\\
	
We know that both $P_{n}(x)$ and $P'_{n}(x)$ do not have common factors because the roots of the orthogonal polynomials are simple. In addition, for $\rho(x)$ defined for each family of classical orthogonal polynomials, we have that both $\rho(x)$ and $P_{n}(x)$ do not share roots because the roots of orthogonal polynomials remain within the range $(a,b)$. In fact:
	\begin{itemize}
		\item[$\clubsuit$] In the Jacobi polynomial, $\rho(x)=1-x^{2}$ whose roots are not in the interval $(-1,1)$
		\item[$\clubsuit$] In the Laguerre polynomials, $\rho(x)=x$ whose root is not in the interval $(0,\infty)$
		\item[$\clubsuit$]  In the Hermite polynomials, $\rho(x)=1$.
	\end{itemize}
	Hence, the polynomial $f(v,x)=\mu vP_{n}(x)+\rho P'_{n}(x)$ is irreducible.\\
	
	On the other hand, using the differential field associated with the differential system and (\ref{sol}), we have that:
	
	$$\begin{array}{rl}
	Xf&=\left(\frac{\lambda_{n}}{\mu}\rho+(\rho'-K_{1}P_{1})v+\mu v^{2}\right)\frac{\partial f}{\partial v}+\rho\frac{\partial f}{\partial x}\\
	&\\
	&=\left(\frac{\lambda_{n}}{\mu}\rho+(\rho'-K_{1}P_{1})v+\mu v^{2}\right)\mu P_{n}+\rho\left(\mu v P'_{n}+\rho'P'_{n}+ \rho^{2}P''_{n}\right)\\
	&\\
	&=\rho'(\mu v P_{n}+\rho P'_{n}) +\mu v(\mu v P_{n}+\rho P'_{n})-\mu K_{1}P_{1}vP_{n}+\rho (\rho P''_{n}+\lambda_{n}P_{n})\\
	&\\
	&=\rho'(\mu v P_{n}+\rho P'_{n}) +\mu v(\mu v P_{n}+\rho P'_{n})-\mu K_{1}P_{1}vP_{n}- K_{1}P_{1}\rho P'_{n}\\
	&\\
	&=\rho'(\mu v P_{n}+\rho P'_{n}) +\mu v(\mu v P_{n}+\rho P'_{n})- K_{1}P_{1}(\mu vP_{n}+\rho P'_{n})\\
	&\\
	&=(\rho'+\mu v- K_{1}P_{1})(\mu vP_{n}+\rho P'_{n})\\
	&\\
	Xf&=(\rho'+\mu v- K_{1}P_{1})f
	\end{array}$$
	
	The above implies that $\mu vP_{n}(x)+\rho P'_{n}(x)=0$ it is an invariant curve for the system (\ref{sispolino})
\end{proof}
The following proposition is entirely a contribution of this paper.
\begin{prop}\label{propid}
The quadratic polynomial system
	\begin{equation}\label{sis darboux}\left\{
	\begin{array}{rl}
	\dot{v}= & \dfrac{\lambda_{n}}{\mu}(1-x^{2})+avx+bv+\mu v^{2}=P(v,x) \\
	\dot{x}= & 1-x^{2} \\
	\end{array}
	\right.
	\end{equation}
	has an invariant of Darboux in the form
	$$I(v,x,t)=\frac{\sqrt{x-1}}{\sqrt{x+1}}e^{t}$$
\end{prop}
\begin{proof}
	The algebraic curves
	$$\begin{array}{lcr}
	f_{1}(v,x)=x+1=0, &  & f_{2}(v,x)=x-1=0
	\end{array}
	$$
	are invariant algebraic curves of the system (\ref{sis darboux}) with cofactors
	$$\begin{array}{lcr}
	K_{1}(v,x)=1-x, &  & k_{2}(v,x)=-1-x,
	\end{array}
	$$
	respectively.\\
	
	In fact, since for this system, the vector field is defined as:
	$$X=P(v,x)\frac{\partial }{\partial v}+(1-x^{2})\frac{\partial }{\partial x}$$
	we obtain that,
	$$\begin{array}{lcr}
	X(f_{1})=(1-x)f_{1} & and & X(f_{2})=(-1-x)f_{2}
	\end{array}
	$$
	
	Now using the theorem  \ref{prop de darboux}, taking $s=1$,

	$$\lambda_{1}K_{1}+\lambda_{1}K_{1}=-1$$
	we obtain 
	$$\begin{array}{lcr}
	\lambda_{1}=-1/2, &  & \lambda_{2}=1/2.
	\end{array}
	$$
	Thus, we obtain the Darboux invariant
	$$I(v,x,t)=\frac{\sqrt{x-1}}{\sqrt{x+1}}e^{t}.$$
\end{proof}

Now we will study the phase portraits on the Poincar\'e disk of the polynomial systems associated with the classical orthogonal polynomials, which is one of the main contributions of this paper.

\begin{prop}
	The phase portrait on the Poincar\'e disk of any quadratic polynomial system
	\begin{equation}\label{retratos1}
	\left\{
	\begin{array}{rl}
	\dot{v}= &\frac{\lambda_{n}}{\mu}(1-x^{2})+avx+\mu v^{2}  \\
	&\\
	\dot{x}= & 1-x^{2} \\
	\end{array}
	\right.
	\end{equation}
	with $\mu\neq0$, $\lambda_{n}>0$ and $a\in \mathbb{R}$
	is topologically equivalent to some of the phase portraits described in the Figure $1$.
\end{prop}

\begin{figure}[h]\label{fig sis 1}
	\begin{center}
		\includegraphics[width=8cm]{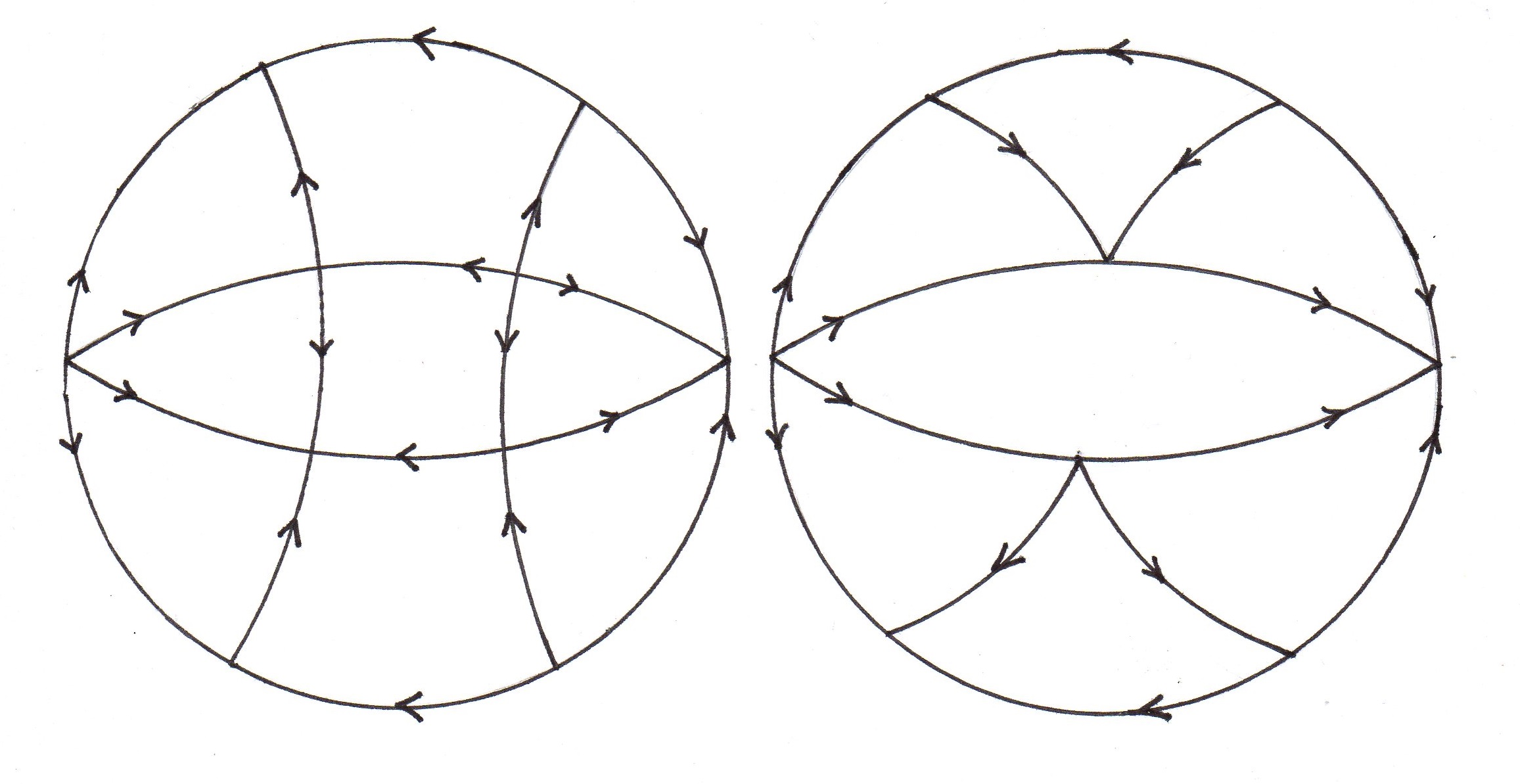}\\
		\caption{Phase portraits for the system \ref{sis retratos1}}\label{fig ejem}
	\end{center}
\end{figure}

\begin{proof}
	In the finite plane, the singular points of the system are
	$$\begin{array}{cccc}
	(0,1), & (0,-1), & (-a/\mu,1), & (a/\mu,-1).\\
	\end{array}
	$$
	Two cases are possibles: If $a\neq 0$ there are four singular points and, if $a=0$ there are only two singular points.\\

	\textbf{Case 1:} $a\neq0$\\
	In the finite plane there are four singular points.
	$$\begin{array}{rl}
	DX(v,x)= & \left[
	\begin{array}{cc}
	ax+2\mu v & -2\frac{\lambda_{n}}{\mu}x+av \\
	&\\
	0 & -2x \\
	\end{array}
	\right]
	\end{array}
	$$
	
	By evaluating this matrix in each of the singular points, we obtain
	$$\begin{array}{cc}
	\begin{array}{rl}
	DX(0,1)=&\left[
	\begin{array}{cc}
	a & -2\lambda_{n}/\mu \\
	0  & -2 \\
	\end{array}
	\right]
	\end{array} & \begin{array}{rl}
	DX(0,-1)=&\left[
	\begin{array}{cc}
	-a & 2\lambda_{n}/\mu\\
	0  & 2 \\
	\end{array}
	\right]
	\end{array} \\
	& \\
	\begin{array}{rl}
	DX(-a/\mu,1)=&\left[
	\begin{array}{cc}
	-a & -(2\lambda_{n}+a^{2})/\mu \\
	0  & -2 \\
	\end{array}
	\right]
	\end{array} & \begin{array}{rl}
	DX(a/\mu,-1)=&\left[
	\begin{array}{cc}
	a & (2\lambda_{n}-a^{2})/\mu \\
	0  & 2 \\
	\end{array}
	\right]
	\end{array}
	\end{array}
	$$
	Therefore, in the finite plane there are two saddle points and two nodes; one stable and the other unstable.\\

	\textbf{Case 2:} $a=0$\\
	In the finite plane there are only two singular points. The Jacobian matrix of the system (\ref{retratos1}), with $a=0$ is:
	$$\begin{array}{rl}
	DX(v,x)= & \left[
	\begin{array}{cc}
	2\mu v & -2\frac{\lambda_{n}}{\mu}x \\
	&\\
	0 & -2x \\
	\end{array}
	\right]
	\end{array}
	$$
	
	$$\begin{array}{cc}
	\begin{array}{rl}
	DX(0,1)= & \left[
	\begin{array}{cc}
	0 & -2\lambda^{n}/\mu \\
	0 & -2 \\
	\end{array}
	\right]
	\end{array}
	& \begin{array}{rl}
	DX(0,-1)= & \left[
	\begin{array}{cc}
	0 & 2\lambda^{n}/\mu \\
	0 & 2 \\
	\end{array}
	\right]
	\end{array}
	\end{array}
	$$

	This is, the singular points $(0,1)$ and $(0,-1)$ are semi-hyperbolics points.\\ 
	
	Using  the theorem  \ref{semi hiper} to be able to analyze the behavior of previous singular points, in a neighborhood of the origin. We must translate these points to the origin of the coordinated plane and after transforming the system to rewrite it in a normal way (the normal forms theorem).\\
	
We perform the following translation; the result will be a system topologically equivalent to (\ref{retratos1}).
	$$\begin{array}{ccc}
	\tilde{x}=x-1, & \ \ & v=v,
	\end{array}
	$$
	
	$$\left\{
	\begin{array}{rl}
	\dot{v}=& \frac{\lambda_{n}}{\mu}(-2\tilde{x}-\tilde{x}^{2})+\mu v^{2}   \\
	\dot{\tilde{x}}=& -2\tilde{x}-\tilde{x}^{2}, \\
	\end{array}
	\right.
	$$
	then,
	$$\begin{array}{ccc}
	\tilde{x}=\tilde{x}, & \ \ & \tilde{v}= v-\dfrac{\lambda_{n}}{\mu}\tilde{x},
	\end{array}
	$$

	$$\left\{
	\begin{array}{rl}
	\dot{\tilde{v}}=& \frac{\lambda_{n}^{2}}{\mu}\tilde{x}^{2}+2 \lambda \tilde{v}\tilde{x}+\mu \tilde{v}^{2} \\
	&\\
	\dot{\tilde{x}}=& -2\tilde{x}-\tilde{x}^{2}.\\
	\end{array}
	\right.
	$$
	
	This last system is topologically equivalent to the system (\ref{retratos1})  and also meets the hypothesis of the theorem for semi-hyperbolic points. \\ If we take
	$$A(\tilde{v},\tilde{x})=\frac{\lambda_{n}^{2}}{\mu}\tilde{x}^{2}+2 \lambda \tilde{v}\tilde{x}+\mu \tilde{v}^{2}$$
	and
	$$B(\tilde{v},\tilde{x})=-\tilde{x}^{2},$$
	
	then $$\tilde{x}=f(\tilde{v})=-\frac{1}{2}\tilde{v}^{2} +o(\tilde{v}^{2})$$
	is the solution of $$-2\tilde{x}+B(\tilde{v},\tilde{x})=0$$
	near of origin.\\
	
	Now,
	$$g(\tilde{v})=A(\tilde{v},f(\tilde{v}))=\mu \tilde{v}^{2}+o(\tilde{v}^{2})$$
	
	because the lowest-order term of the function $g(\tilde{v})$ is even, the singular point $(0,1)$ is a saddle-node point.\\
	
	Now, for the semi-hyperbolic point $(0,-1)$ we make the transformations
	$$\begin{array}{cc}
	v= v,   & \tilde{x}= x+1 \\
	&  \\
	\tilde{x}= \tilde{x} & \tilde{v}=v-\frac{\lambda_{n}}{\mu}\tilde{x}
	\end{array}
	$$
	obtaining that $(0,-1)$ is a saddle-node point.\\

	Now we will analyze the singular points in infinity, using the transformations on the Poincar\'e sphere, see \cite{PK} .\\
	The flow defined by the study system \ref{retratos1}, on the equator of the Poincar\'e sphere except $(\pm 1,0,0)$ is topologically equivalent to the flow defined by the system
	
	$$\left\{
	\begin{array}{rl}
	\dot{v}=&-\lambda_{n}/\mu+(a+1)v+\mu v^{2}+\lambda_{n}/\mu z^{2}-vz^{2}  \\
	
	\dot{x}= & -z^{3}+z \\
	\end{array}
	\right.
	$$
	
whose singular points to study are:
	
	$$\begin{array}{ccc}
	( v_{1},0)= \left(\frac{-(a+1)+\sqrt{(a+1)^{2}+4\lambda_{n}}}{2\mu},0\right) &  & (v_{2},0)=\left(\frac{-(a+1)-\sqrt{(a+1)^{2}+4\lambda_{n}}}{2\mu},0\right)
	\end{array}
	$$
	
	$$\begin{array}{rl}
	DX(v,z)= & \left[
	\begin{array}{cc}
	(a+1)+2\mu -z^{2}- v & 2\frac{\lambda_{n}}{\mu}z -2vz \\
	&\\
	0 & -3z^{2}+1 \\
	\end{array}
	\right]
	\end{array}$$
	then,
	$$\begin{array}{cc}
	\begin{array}{rl}
	DX(v_{1},0)= & \left[
	\begin{array}{cc}
	\sqrt{(a+1)^{2}+4\lambda_{n}} & 0 \\
	0 & 1 \\
	\end{array}
	\right]
	\end{array}
	& \begin{array}{rl}
	DX(v_{2},0)= & \left[
	\begin{array}{cc}
	-\sqrt{(a+1)^{2}+4\lambda_{n}} & 0 \\
	0 & 1 \\
	\end{array}
	\right]
	\end{array}
	\end{array}
	$$
which indicates that, $(v_{1},0)$ is an unstable node and $(v_{2},0)$ is a saddle point.\\
	The flow defined by the study system on, the equator of the Poincar\'e sphere except $(0,\pm 1,0) $ is topologically equivalent to the flow defined by the system
	$$\left\{\begin{array}{rl}
	\dot{x}=& -\frac{\lambda_{n}}{\mu}xz^{2}+\frac{\lambda_{n}}{\mu}x^{3}-\mu x+z^{2}-(a+1)x^{2}\\
	&\\
	\dot{z}=& -\frac{\lambda_{n}}{\mu}z^{3}+\frac{\lambda_{n}}{\mu}x^{2}z-axz-\mu z
	\end{array}
	\right.$$
	In which it is only necessary to study the behavior of the singular point, the origin.
	
	$$\begin{array}{rl}
	DX(x,z)= & \left[
	\begin{array}{cc}
	-\frac{\lambda_{n}}{\mu}z^{2}+3\frac{\lambda_{n}}{\mu}x^{2}-\mu-2(a+1)x & -2\frac{\lambda_{n}}{\mu}xz+2z \\
	&\\
	2\frac{\lambda_{n}}{\mu}xz-az & -3\frac{\lambda_{n}}{\mu}z^{2}+\frac{\lambda_{n}}{\mu}x^{2}-ax-\mu \\
	\end{array}
	\right]
	\end{array}$$
	Evaluating this matrix in $(0,0)$
	$$\begin{array}{rl}
	DX(0,0)= & \left[
	\begin{array}{cc}
	-\mu & 0 \\
	0 & -\mu \\
	\end{array}
	\right]
	\end{array}
	$$
	This is, the origin of this last system is a node and its stability depends on the sign of $\mu$.\\
	
\end{proof}

\begin{obs}
	For specific values of the parameter $a$, phase portraits are obtained for the polynomial systems associated with the following orthogonal polynomials:
	$$\begin{array}{lcl}
	a=0, & \ \ & P_{n} \\
	a=-1, & \ \ & T_{n} \\
	a=1 & \ \ &  U_{n}\\
	a=2\alpha-1 & \ \ &C_{n}^{(\alpha)}
	\end{array}$$
\end{obs}

\begin{prop}
	The phase portrait on the poincar\'e disk of any quadratic polynomial system
	\begin{equation}\label{sis retratos2}
	\left\{
	\begin{array}{rl}
	\dot{v}= &\frac{\lambda_{n}}{\mu}x+av+bvx+\mu v^{2}  \\
	&\\
	\dot{x}= & x \\
	\end{array}
	\right.
	\end{equation}
	with $\mu\neq0$, $\lambda_{n}>0$ and $a,b\in \mathbb{R}$, is topologically equivalent to some of the phase portraits described in the Figure \ref{figsis2}
\end{prop}

\begin{figure}[h!]
	\begin{center}
		\includegraphics[width=6cm]{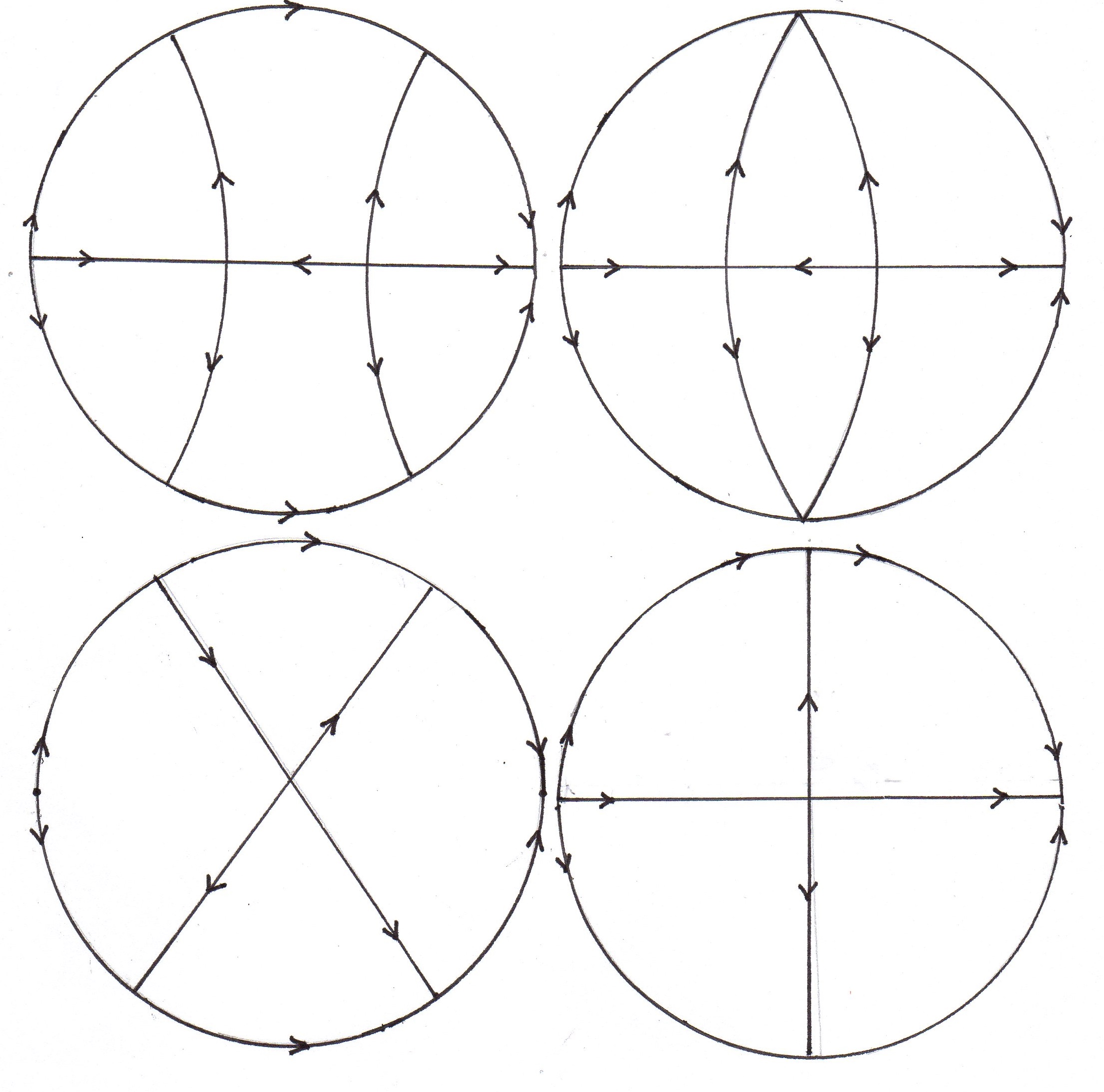}\\
		\caption{Portraits of phase for \ref{sis retratos2}}\label{figsis2}
	\end{center}
\end{figure}

\begin{proof}
	In this system the singular points in the finite plane have the form $(0,0)$ and $(\dfrac{-a}{\mu},0)$. this is, if $a=0$ there is only one singular point and if $a\neq0$ there are two singular points.\\
	The Jacobian Matrix of the system is:
	$$DX(v,x)=\left[
	\begin{array}{cc}
	a+bx+2\mu v & \dfrac{\lambda_{n}}{\mu}x+bv \\
	&\\
	0 & 1 \\
	\end{array}
	\right]
	$$
	\textbf{Case 1:} \emph{Laguerre associate} $a\neq0$.\\
	$$\begin{array}{cc}
	DX(0,0)=\left[
	\begin{array}{cc}
	a & 0 \\
	0 & 1 \\
	\end{array}
	\right]
	&  DX(-a/\mu,0)=\left[
	\begin{array}{cc}
	-a & -ab/\mu \\
	0 & 1 \\
	\end{array}
	\right]
	\end{array}
	$$
	Indistinct of the sign of $a$, in the finite plane there is a saddle point and an unstable node.\\
	
	\textbf{Case 2:} \emph{Laguerre} $a=0$.\\
	$$DX(0,0)=\left[
	\begin{array}{cc}
	0 & 0 \\
	0 & 1 \\
	\end{array}
	\right]$$
	This implies that the origin is a singular semi-hyperbolic point. \\
	Making the transformations:
	$$\begin{array}{cc}
	\tilde{v}=v-\dfrac{\lambda_{n}}{\mu}x, & x=x
	\end{array}
	$$
	we get the following system, which is, topologically equivalent to (\ref{sis retratos2})
	$$\left\{
	\begin{array}{rl}
	\dot{\tilde{v}}=&\dfrac{\lambda_{n}(b+\lambda_{n})}{\mu}x^{2}+(b+2\lambda_{n})\tilde{v}x+\mu v^{2} \\
	& \\
	\dot{x}=&x
	\end{array}
	\right.
	$$
Applying the theorem for semi-hyperbolic points, we take
	$$\begin{array}{lcr}
	A(\tilde{v},x)=\dfrac{\lambda_{n}(b+\lambda_{n})}{\mu}x^{2}+(b+2\lambda_{n})\tilde{v}x+\mu v^{2}&\ \ & B(\tilde{v},0)=0.
	\end{array}
	$$
	Then $x=f(\tilde{v})=0$ is the solution of equation $x+B(\tilde{v},0)=0$, in a neighborhood of origin.\\
	Now,
	$$g(\tilde{v})=A(\tilde{v},0)=\mu \tilde{v}^{2}+o(\tilde{v}^{2}).$$
	therefore, the origin is a saddle-node.\\

	Again the singular points in infinity will be analyzed, using the transformations on the poincar\'e sphere. \\
	The flow defined by the study system \ref{sis retratos2}, on the equator of the Poincaré sphere except $(\pm 1,0,0)$ is topologically equivalent to the flow defined by the system
	$$\left\{
	\begin{array}{rl}
	\dot{v}=& \dfrac{\lambda_{n}}{\mu}z+bv+(a-1)vz+\mu v^{2} \\
	&\\
	\dot{z}= & -z^{2} \\
	\end{array}
	\right.
	$$
	whose singular points are: $(0,0)$ and $(-b/\mu,0)$. If $b\neq0$ there are two singular points, If $b=0$ there is only one singular point.\\
	The Jacobian matrix associated with this last system is:
	\begin{equation}\label{retrato2infin}DX(v,z)=\left[
	\begin{array}{ccc}
	b+(a-1)z+2v\mu & \ \ &\dfrac{\lambda_{n}}{\mu}+(a-1)v \\
	0 & \ \ & -2z \\
	\end{array}
	\right]\end{equation}
	\textbf{Case 1:} \emph{Laguerre and Laguerre associate} $b\neq0$.\\
	$$\begin{array}{lr}
	DX(0,0)=\left[
	\begin{array}{ccc}
	b & \ \ &\dfrac{\lambda_{n}}{\mu} \\
	0 & \ \ & 0 \\
	\end{array}
	\right]  &
	DX(-b/\mu,0)=\left[
	\begin{array}{ccc}
	-b & \ \ &\dfrac{\lambda_{n}+b(1-a)}{\mu} \\
	0 & \ \ & 0 \\
	\end{array}.
	\right]
	\end{array}
	$$
	This is, $(0,0)$ and $(-b/\mu,0)$ they are semi-hyperbolic points.\\
	To express the system (\ref{retrato2infin}) in the canonical form, and thus be able to apply the theorem for semi-hyperbolic points; we perform the following transformations:
	$$\begin{array}{ccc}
	\tilde{v}=\dfrac{\lambda_{n}}{\mu}z+bv & \ \ & z=z
	\end{array}
	$$
	obtaining the following system, which is topologically equivalent to (\ref{retrato2infin})
	$$\left\{
	\begin{array}{rl}
	\dot{\tilde{v}}= & b\tilde{v}+\dfrac{\lambda_{n}(-a+\lambda_{n}/b)}{\mu}z^{2}+(a-1-2\lambda_{n}/b)\tilde{v}z +\dfrac{\mu}{b}\tilde{v}^{2}\\
	&\\
	\dot{z}= & -z^{2} \\
	\end{array}
	\right.
	$$
	where
	$$\begin{array}{ccc}
	A(\tilde{v},z)=-z^{2} & \ \ & B(\tilde{v},z)=\dfrac{\lambda_{n}(-a+\lambda_{n}/b)}{\mu}z^{2}+(a-1-2\lambda_{n}/b)\tilde{v}z +\dfrac{\mu}{b}\tilde{v}^{2}
	\end{array}
	$$
	Let $\tilde{v}=f(z)$ the solution of equation $b\tilde{v}+B(\tilde{v},z)=0$ in a neighborhood of origin. then
	$$g(z)=A(f(x),z)=-z^{2}$$
	So $(0,0)$ is a saddle-node.\\
	
	For the point $(-b/\mu,0)$ we will successively the following transformations
	$$\begin{array}{ccc}
	\tilde{v}=v+\dfrac{b}{\mu} & \ \ & z=z \\
	& \ \ &  \\
	\bar{v}=\dfrac{\lambda_{n}+b(1-a)}{\mu}z-b\tilde{v} & \ \ & z=z
	\end{array}
	$$
	obtaining the system topologically equivalent to \ref{retrato2infin}
	$$\left\{
	\begin{array}{rl}
	\dot{\bar{v}}= & -b\bar{v}+B(\bar{v},z) \\
	&  \\
	\dot{z}= & -z^{2} \\
	\end{array}
	\right.
	$$
	where
	$$B(0,0)=DB(0,0)=0$$
	and
	$$A(\bar{v},z)=-z^{2}.$$
	Let $\bar{v}=f(z)$ the solution of the equation $-b\bar{v}+B(\bar{v},z)=0$ in a neighborhood of the origin of this latter system. Then
	$$g(z)=A(f(z),z)=-z^{2}.$$
	Therefore, the point $(-b/\mu,0)$ is a saddle-node.\\

	\textbf{Case 2:} $b=0$.\\
	$$DX(0,0)=\left[
	\begin{array}{ccc}
	0 & \ \ &\dfrac{\lambda_{n}}{\mu} \\
	0 & \ \ & 0 \\
	\end{array}
	\right] $$
	That is, the origin is a unique nilpotent point for this system. \\
	We make the transformation
	$$\begin{array}{lcr}
	\tilde{v}=\dfrac{\mu}{\lambda_{n}}v, & \ \ & z=z
	\end{array}
	$$
	obtaining the system topologically equivalent to the system (\ref{retrato2infin}):
	$$\left\{
	\begin{array}{rl}
	\dot{\tilde{v}}= & z+(a-1)\tilde{v}z+\lambda_{n}\tilde{v}^{2}  \\
	&   \\
	\dot{z}= & -z^{2}  \\
	\end{array}
	\right.
	$$
This last system fulfills the conditions of theorem for singular nilpotent points 
where
	$$\begin{array}{lcr}
	A(\tilde{v},z)=(a-1)\tilde{v}z+\lambda_{n}\tilde{v}^{2} & \ \  & B(\tilde{v},z)= -z^{2}
	\end{array}
	$$
	Otherwise, $z=f(\tilde{v})=(1-\lambda_{n}-a)\tilde{v}^{2}+0(\tilde{v}^{2})$ is the solution to the equation $$z+A(\tilde{v},z)=0$$ in a neighborhood of the origin.\\
	Then,
	$$\begin{array}{rl}
	F(\tilde{v})=&B(\tilde{v},f(\tilde{v}))=-(1-\lambda_{n}-a)^{2}\tilde{v}^{4}+o(\tilde{v}^{4}) \\
	&\\
	G(\tilde{v})=&\left(\dfrac{\partial A}{\partial \tilde{v}}+\dfrac{\partial B}{\partial z}\right)(\tilde{v},f(\tilde{v}))=2\lambda_{n}\tilde{v}+o(\tilde{v})
	\end{array}
	$$
	In this case $m=4$ y $n=1$. Since $m$ is even and $m>2n+1$ the origin is a saddle-node.\\

	For the infinity, the flow defined by the system, on the equator Poincar\'e sphere excepting $(0,\pm 1,0)$ is topologically equivalent to the flow defined by the system
	$$\left\{\begin{array}{rl}
	\dot{x}=& (1-a)xz-\frac{\lambda_{n}}{\mu}x^{2}z-bx^{2}-\mu x\\
	&\\
	\dot{z}=& -\frac{\lambda_{n}}{\mu}xz^{2}-az^{2}-xz-\mu z,
	\end{array}
	\right.$$
	in which it is only necessary to study the behavior of the singular point, the origin.
	$$\begin{array}{rl}
	DX(x,z)= & \left[
	\begin{array}{cc}
	(1-a)z-2\frac{\lambda_{n}}{\mu}xz-2bx-\mu  & (1-a)x-\frac{\lambda_{n}}{\mu}x^{2} \\
	&\\
	-\frac{\lambda_{n}}{\mu}z^{2}-z &  -2\frac{\lambda_{n}}{\mu}xz-2az-x-\mu  \\
	\end{array}.
	\right]
	\end{array}$$
	In $(0,0)$
	$$\begin{array}{rl}
	DX(0,0)= & \left[
	\begin{array}{cc}
	-\mu & 0 \\
	0 & -\mu \\
	\end{array}
	\right]
	\end{array}
	$$
that is, the origin of this last system is a node and its stability depends on the sign of $\mu$.

\end{proof}

\begin{obs}
	In the previous proposition, for specific values of the parameters $a$ and $b$, the phase portraits for the polynomial systems associated with the following orthogonal polynomials are obtained:
	$$\begin{array}{lccl}
	a=0,&b=1 & \ \ &L_{n}  \\
	a=-\alpha,&b=1 & \ \ &L_{n}^{(\alpha)}
	\end{array}$$
\end{obs}
To finish this section we compute the differential Galois group and the elements of Darboux integrability to the quadratic polynomial vector field related with Chebyshev differential equation.
\begin{prop}
	For the Chebyshev differential equation
	\begin{equation}\label{ecu. chebychef}y''-\dfrac{x}{1-x^{2}}y'+\dfrac{\lambda_{n}}{1-x^{2}}y=0\end{equation}
	where $\lambda_{n}=n^{2}$, $n\in\mathbb{N}$,
	the following statement are true:
	\begin{enumerate}
		\item $G(L/K)$ of the Chebyshev equation is isomorphic to $\mathbb{Z}_{2}$, where $K=\mathbb{C}(x)$
		\item The first integrals of the fields
		\begin{equation}\label{sistema redu}\left\{
		\begin{array}{rl}
		\dot{w}= & -2-4\lambda_{n}+(4\lambda_{n}-1)x^{2}-4(1-x^{2})^{2}w^{2} \\
		&  \\
		\dot{x}= & 4(1-x^{2})^{2} \\
		\end{array}
		\right.
		\end{equation}
		and
		\begin{equation}\label{sis asociado tchebichef}\left\{
		\begin{array}{rl}
		\dot{v}= & \dfrac{\lambda_{n}}{\mu}(1-x^{2})-xv+\mu v^{2} \\
		&  \\
		\dot{x}= &1-x^{2} \\
		\end{array}
		\right.
		\end{equation}
		associated with the Chebyshev equation, are:
		$$\begin{array}{rl}
		I(w,z)= &\dfrac{-w+\dfrac{U_{n-1}'}{U_{n-1}}-\dfrac{3x}{2(1-x^{2})}}{-w+\dfrac{T_{n}'}{T_{n}}-\dfrac{x}{2(1-x^{2})}}\cdot\dfrac{U_{n-1}}{T_{n}}\sqrt{1-x^{2}} \\
		& and  \\
		I(v,z)= &\dfrac{U_{n-1}'(1-x^{2})+U_{n-1}(\mu v-x)}{T'_{n}(1-x^{2})+\mu T_{n}vx}\cdot\sqrt{1-x^{2}}
		\end{array}
		$$
		
	\end{enumerate}
\end{prop}

\begin{proof}
	\begin{enumerate}
		\item It is known that $y_{1}=T_{n}$ , $y_{2}=U_{n-1}\sqrt{1-x^{2}}$ are two linearly independent solutions of the equation (\ref{ecu. chebychef}). If we take the differential body $K=\mathbb{C}(x)$ of all the rational functions of variable $x$, we consider the extension of the body $L=K[\sqrt{1-x^{2}}]$. To calculate the differential Galois group of the equation
		(\ref{ecu. chebychef}) all differential automorphisms in the extension must be calculated $L$. That is, find a matrix $$A_{\phi}=\left[
		\begin{array}{cc}
		a & b \\
		c & d \\
		\end{array}
		\right]
		$$
		such that
		$$\phi\left[
		\begin{array}{c}
		y_{1}\\
		y_{2} \\
		\end{array}
		\right]=A_{\phi}\left[
		\begin{array}{c}
		y_{1}\\
		y_{2} \\
		\end{array}.
		\right]
		$$
		By matrix operations we have:
		$$\begin{array}{lcr}
		\phi(y_{1})=ay_{1}+by_{2}, & \  \ &  \phi(y_{2})=cy_{1}+dy_{2}
		\end{array}
		$$
		On the other hand, $y_{1},y_{2}\in \mathbb{C}(x)$ and $\phi$ are automorphisms,then we get
		$$\begin{array}{ccc}
		\phi(y_{1})=y_{1}, & \ \ & \phi(y_{2})=cy_{2}
		\end{array}
		$$
		when $c^{2}=1$. Then we can conclude that
		$$A_{\phi}=\left[
		\begin{array}{cc}
		1 & 0 \\
		0 & c \\
		\end{array}
		\right]
		$$
		This is,
		$$D(L/K)\cong\left\{\left[
		\begin{array}{cc}
		1 & 0 \\
		0 & 1 \\
		\end{array}
		\right],\left[
		\begin{array}{cc}
		1 & 0 \\
		0 & -1 \\
		\end{array}
		\right]
		\right\}\cong \mathbb{Z}_{2}$$
		\item If in the equation (\ref{ecu. chebychef}) we consider $b_{1}(x)=\dfrac{-x}{1-x^{2}}$ y $b_{0}(x)=\dfrac{\lambda_{n}}{1-x^{2}}$, then transformation $$y=ze^{-\frac{1}{2}\int b_{1}(x)dx}$$ allows us to obtain the reduced second order equation
		\begin{equation}\label{redu de 2 grado}
		z'' =\left(\dfrac{-2-x^{2}-4\lambda_{n}(1-x^{2})}{4(1-x^{2})^{2}}\right)z,
		\end{equation}
		with $$z=y(1-x^{2})^{\frac{1}{4}}.$$
Since $y_{1}=T_{n}$ and $y_{2}=U_{n-1}$ are linearly independent solutions of the Chebyshev equation, then:
		$$\begin{array}{lcr}
		z_{1}=T_{n}(1-x^{2})^{\frac{1}{4}}, & \  \ & z_{2}=U_{n-1}(1-x^{2})^{\frac{3}{4}},
		\end{array}
		$$
		are linearly independent solutions of the reduced second order equation (\ref{redu de 2 grado}). \\
		
		On the other hand, the differential equation associated with the system (\ref{sistema redu}) have the form:
		
		$$w'=\dfrac{ -2-4\lambda_{n}+(-1+4\lambda_{n})x^{2}}{4(1-x^{2})^{2}}-w^{2}$$
		
		and applying the transformation $w=\dfrac{z'}{z}$, is equivalent to the equation (\ref{redu de 2 grado}). from this the solutions of this last equation are given by:
		
		$$\begin{array}{c}
		w_{1}=\dfrac{z'_{1}}{z_{1}}=(lnz_{1})'=\dfrac{T'_{n}}{T_{n}}-\dfrac{x}{2(1-x^{2})}\\
		and\\
		w_{2}=\dfrac{z'_{2}}{z_{2}}=(lnz_{2})'=\dfrac{U'_{n-1}}{U_{n-1}}-\dfrac{3x}{2(1-x^{2})}.
		\end{array}
		$$
		Then. by Lemma 1 of \cite{acpan}, we get that the first integral of the system (\ref{sistema redu}) have the form:
		$$I(w,x)=\dfrac{-w(x)+w_{2}(x)}{-w(x)+w_{1}(x)}\cdot e^{\left(\int(w_{2}(x)-w_{1}(x))dx\right)}.$$
		This is,
		$$I(w,x)=\dfrac{-w+\dfrac{U_{n-1}'}{U_{n-1}}-\dfrac{3x}{2(1-x^{2})}}{-w+\dfrac{T_{n}'}{T_{n}}-\dfrac{x}{2(1-x^{2})}}\cdot\dfrac{U_{n-1}}{T_{n}}\sqrt{1-x^{2}}$$
		
		Now to find the first integral of the system (\ref{sis asociado tchebichef}),  can be noticed that the foliation of this system and the foliation of the system (\ref{sistema redu}) are
		$$\begin{array}{rl}
		v'=& \dfrac{\lambda_{n}}{\mu}-\dfrac{x}{1-x^{2}}v+\dfrac{\mu}{1-x^{2}}v^{2}, \\
		& \\
		w'=&\dfrac{-2-4\lambda_{n}+(4\lambda_{n}-1)x^{2}}{4(1-x^{2})^{2}}-w^{2}.
		\end{array}
		$$
		which are related through the transformation
		$$w=-\dfrac{x}{2(1-x^{2})}-\dfrac{\mu v}{1-x^{2}}.$$
		Therefore, replacing we obtained 
		
		$$I(v,x)=\dfrac{\dfrac{x}{2(1-x^{2})}+\dfrac{\mu v}{1-x^{2}}+\dfrac{U_{n-1}'}{U_{n-1}}-\dfrac{3x}{2(1-x^{2})}}{\dfrac{x}{2(1-x^{2})}+\dfrac{\mu v}{1-x^{2}}+\dfrac{T_{n}'}{T_{n}}-\dfrac{x}{2(1-x^{2})}}\cdot\dfrac{U_{n-1}}{T_{n}}\sqrt{1-x^{2}}$$
	after simplifying we get the first integral described for the system
		(\ref{sis asociado tchebichef}).
	\end{enumerate}
\end{proof}

\section*{Final Remarks}
In this paper we studied algebraically through differential Galois theory and Darboux theory of integrability, as well qualitatively through the analysis of critical points, some quadratic polynomial vector fields related with classical orthogonal polynomials.

\section*{Acknowledgements}
The author thank to Camilo Sanabria and Dmitri Karp by their useful comments and suggestions.


\end{document}